\documentclass[letterpaper,10pt,conference]{ieeeconf}  

\IEEEoverridecommandlockouts                              
\usepackage{amsmath}
\usepackage{amssymb}
\usepackage{algorithm}  
\usepackage{algorithmicx}  
\usepackage{algpseudocode}  
\usepackage{graphicx}
\usepackage{float}
\usepackage{frenchineq}

\usepackage{hyperref}
\usepackage{cleveref}
\usepackage[percent]{overpic}

\usepackage[caption=true,font=scriptsize]{subfig}
\usepackage[colorinlistoftodos,bordercolor=orange,backgroundcolor=orange!20,linecolor=orange,textsize=scriptsize]{todonotes}

\newcommand{\R}{\mathbb{R}}

\newcommand{\Rm}{\R_{\max}}

\newcommand{{\Mw}}{\mathcal{M}}

\newcommand{\X}{X}

\DeclareMathOperator*{\argmax}{arg\,max}

\makeatletter
\let\NAT@parse\undefined
\makeatother
\usepackage[numbers]{natbib}

\newenvironment{skproof}{\noindent\hspace{2em}{\itshape Sketch of Proof: }}{\hspace*{\fill}~\QED\par\endtrivlist\unskip}

\newtheorem{theorem}{Theorem}[section]
\newtheorem{proposition}{Proposition}[section]

\newtheorem{assumption}{Assumption}[section]

\newtheorem{remark}{Remark}[section]

\title{\LARGE \bf
	Tropical low-rank approximation and application to optimal control of N-body systems
}

\author{Marianne Akian, St\'ephane Gaubert, Shanqing Liu and Yang Qi
	\thanks{Marianne Akian, St\'ephane Gaubert are with Inria and CMAP, \'Ecole polytechnique,  Institut Polytechnique de  Paris, CNRS, 
	  {\tt\footnotesize \{Marianne.Akian; Stephane.Gaubert\}@inria.fr}}%
	\thanks{Shanqing Liu is with Division of Applied Mathematics, Brown University, {\tt\footnotesize Shanqing\_Liu@brown.edu}}%
        \thanks{This work was partly done when Shanqing Liu and Yang Qi were with Inria and CMAP, \'Ecole polytechnique,  Institut Polytechnique de  Paris, CNRS, {\tt\footnotesize yang.qi@inria.fr}}
        }
\begin{document}
	
	\maketitle
	\thispagestyle{empty} 
	\pagestyle{empty}
	\begin{abstract} 
	  We study the approximation of the value function of deterministic optimal control problems with fixed initial state, motivated by \(N\)-body systems. In this setting, the action functional consists of local kinetic and potential terms, along with an interaction potential. We exploit this structure to approximate the value function using a tropical tensor of small rank, i.e.\ a supremum of a small number of additively separable functions. We propose a trajectory-based tropical low-rank approximation method. Rather than propagating basis functions globally,  as in usual tropical numerical methods, the approximation of the value function is improved only along a sequence of relevant trajectories. The resulting approximations form a monotone family of computable lower bounds for the exact value function, with the tropical tensor rank increasing at most linearly with the number of outer iterations. Under suitable regularity assumptions, we show that at the initial state, and also at the optimal trajectory starting from this state, the lower bounds converge to the exact value. In the $N$-body setting, the generated basis functions remain additively separable across subsystems, thereby yielding a structured tropical low-rank approximation. Numerical experiments on $N$-body systems with Coulomb-type repulsion illustrate the effectiveness of the approach up to state dimension \(200\), within a  half hour time budget.

	\end{abstract}

	\section{INTRODUCTION}
	
	\subsection{Motivation and Context}
	
	In this work, we consider the numerical approximation of the value function of a deterministic optimal control problem. By the \emph{dynamic programming principle}~\cite{flemingsoner,bardi2008optimal}, the value function is characterized as the \emph{viscosity solution} of a first-order \emph{Hamilton--Jacobi--Bellman} (HJB) equation~\cite{crandall1983viscosity,crandall1984some}. Compared with local optimality conditions such as \emph{Pontryagin's maximum principle}~\cite{raymond1998pontryagin,raymond1999pontryagin}, this approach has the major advantage of characterizing the global optimum. However, one of its main limitations is the \emph{curse-of-dimensionality}, that is  solving the Bellman equation or the associated nonlinear recurrence on the whole state space typically requires grid-based discretizations whose complexity grows exponentially with the dimension~\cite{Crandall1984TwoAO,falcone2013semi}.
	
	This difficulty is particularly severe for $N$-body systems, whose state variable takes the form
	\[
	x=(x_1,\dots,x_N)\in (\mathbb{R}^d)^N,
	\]
	so that the state dimension is \(n=Nd\). The present work is motivated by the observation that, in such systems, the action functional 
        associated with the least action principle, often exhibits a structured \textit{tropical} separability. 
        More precisely, 
        if the interaction energy admits a tropical (max-plus) tensor low-rank approximation, of the form
        \[
	G(x)=\max_{r=1,\dots,R}\bigl(g_1^r(x_1)+\cdots + g_N^r(x_N)\bigr),
      \]
	then the corresponding one-step dynamic programming update naturally preserves a max-plus representation by additively separable basis functions. This suggests a tropical analogue of classical low-rank approximation, adapted to the product structure of $N$-body systems.
	
	Such a viewpoint is closely related to max-plus (tropical) numerical  methods for optimal control and HJB equations~\cite{fleming2000max,Mc2006,akian2008max,lakhouathesis,Mc2007,qu2014contraction,akian2025stochastic,darbon2023neural}. These methods exploit the max-plus linearity of the HJB evolution semigroup, namely the \emph{Lax--Oleinik semigroup}~\cite{maslov1987methodes}, and represent the value function as a supremum of basis functions, which is then propagated in time. 
	
	 In analogy with classical tensor low-rank approximation, where a high-dimensional object is represented by a small number of separable terms, tropical tensor low-rank approximation represents a function as the supremum of a small number of additively separable terms~\cite{saadi2021zero,akian2023tropical}. In the $N$-body setting, the simplest such objects are tropical tensor rank-one basis functions, obtained as sums of terms depending on each subsystem separately.
	
	 However, a direct propagation of tropical tensor low-rank representations over time leads to a rapid growth of rank, even when the initial representation is low-rank, the number of basis functions typically increases exponentially with the number of time steps. This is the so-called \textit{curse-of-complexity} in curse of dimensionality free tropical numerical methods developed by McEneaney~\cite{Mc2007,mceneaney2008curse} 
         which was overcomed there by alternating propagation and pruning steps.
         Another way to overcome the curse-of-complexity proposed by Qu~\cite{Qu2014} and then in \cite{akian2025stochastic} is to replace pruning by sampling. This technique was also used in  \cite{akian:hal-01675067} in the context of stochastic control problems.

         Here we rather take inspiration from trajectory-based dynamic programming approaches used in stochastic dual dynamic programming (SDDP) method for solving convex multistage stochastic optimization problems~\cite{pereira1991multi,shapiro2011analysis,girardeau2015convergence,guigues2020inexact}, or in point based methods for solving partially observable Markov Decision Processes \cite{shani2013survey}.                           Rather than propagating basis functions over the whole state space, we propagate and enrich the approximation only along a sequence of relevant trajectories.
         A related idea was considered in  \cite{akian2025stochastic}, for tropical approximations using random trajectories.
	Related ideas that focus on approximating value functions along one or several trajectories also appear in tree-based discretizations~\cite{alla2019efficient,alla2020tree}, adaptive control discretizations~\cite{bokanowski2022optimistic}, and multilevel methods~\cite{akian2023multi,akian2023adaptive}.

	The central idea of the present work is therefore to employ tropical tensor low-rank approximation along with a deterministic trajectory-based propagation. In this way, we retain a functional approximation of the value function by computable lower bounds and an associated approximation of optimal trajectory, while avoiding the full combinatorial explosion of basis functions that would arise in a global propagation.

        Dower and McEneaney previously applied max-plus methods to the representation of solutions of $N$-body problems from classical mechanics~\cite{McEneaney2015}, reducing the computation of fundamental solutions to differential games. In contrast, our method differs by its numerical motivation and by its trajectory-based and low-rank approximation nature. 

	\subsection{Contribution}
	
	In this paper, we approximate the value function of a deterministic optimal control problem with fixed initial state, together with an associated optimal trajectory, by means of a trajectory-based tropical tensor low-rank approximation. We consider a maximization problem in which the running reward is a semiconvex function of the state variable.
	Under suitable assumptions, we show that semiconvexity is propagated by the discrete Bellman equation, yielding that at each time, the value function is the supremum of (general) concave quadratic functions with the same curvature, that play the role of 
        lower supporting (quadric) hypersurfaces for (the epigraph of) the value function. 
	Based on this, we propose an iterative backward-forward algorithm which, starting from a reference trajectory, enriches a lower approximation by adding one quadratic function at each time step and at each outer iteration.
        The added quadratic functions correspond to supporting hypersurfaces
         of the approximate value function at the trajectory points.
        Then, the trajectory is updated by choosing an optimal trajectory starting in the fixed initial state and associated to the new approximation of the value function. 
	The resulting approximations of the value function form a monotone sequence of computable lower bounds, while the tropical tensor rank grows at most linearly with the number of outer iterations.
	
	We then apply this framework to $N$-body systems. 
	In this setting, the generated basis functions are additively separable across subsystems, hence they have tropical tensor rank one. This yields a structured tropical tensor low-rank approximation of the value function. Numerical experiments  with Coulomb-type interaction illustrate the approach up to dimension \(200\), within a
 half hour        time budget.
	
	The paper is organized as follows. \Cref{sec-pre} recalls preliminaries on optimal control, tensor low-rank approximation, and tropical (max-plus) algebra. \Cref{sec-lowrank} introduces tropical tensor low-rank approximation and its connection with $N$-body systems. \Cref{section-algo} presents the trajectory-based iterative algorithm. \Cref{numerics_nbody} reports numerical results.

	\section{PRELIMINARIES}\label{sec-pre}

  \subsection{Max-plus algebra and optimal control}
  The max-plus or tropical semifield is the set $\R\cup\{-\infty\}$ of
  real numbers with $-\infty$, endowed with the laws 
  $ a\oplus b:=\max\{a,b\}$ and $a\odot b:=a+b$ 
  as addition and multiplication. This semifield will be denoted by
  $\R_{\max}$. 

  Max-plus algebra appears naturally in deterministic optimal control. 
  Consider the finite-horizon problem
  \begin{equation*}\label{eq:ocp}
  	v(x,t):=\sup_{u(\cdot)}\left\{
  	\int_t^T \ell(x^u(s),u(s))\,ds+\phi(x^u(T))
  	\right\},
  \end{equation*}
  where $x^u:[0,T]\to \R^n$ solves
  \[
  \dot x^u(s)=f(x^u(s),u(s)),
  \qquad x^u(t)=x .
  \]
  Under standard assumptions, the value function $v$ is the viscosity solution of the HJB equation
  \begin{equation}\label{eq:HJB}
  	\left\{ 
  	\begin{aligned}
  		&-\partial_t v(x,t)-H(x,Dv(x,t))=0, \\ 
  		& v(x,T)=\phi(x),
  	\end{aligned}
  	\right.
  \end{equation}
  with Hamiltonian
  \[
  H(x,p):=\sup_{u\in U}\{p\cdot f(x,u)+\ell(x,u)\}.
  \]
  The associated Lax--Oleinik semigroup $S^\tau$ of~\eqref{eq:HJB}, i.e. the evolution semigroup of this PDE, is max-plus linear, meaning that for all
  (smooth) functions $\phi,\psi:\R^n\to \Rm$ and scalars $\lambda\in\Rm$, we have 
  \begin{equation*}\label{eq:mplinear}
  	\left\{
  	\begin{aligned}
  		& S^\tau(\phi \oplus \psi)=S^\tau\phi \oplus S^\tau\psi, \\ 
  		& S^\tau(\lambda\odot \phi)=\lambda \odot S^\tau\phi.
  	\end{aligned}
  	\right.
  \end{equation*}
  Here $\phi\oplus \psi$ is the pointwise max-plus addition of functions $\phi$ and $\psi$, which coincides with the supremum of these functions for the usual partial order.
  The above max-plus linearity property is one of the foundations of max-plus numerical methods for Hamilton--Jacobi equations~\cite{fleming2000max,akian2008max}.

    \subsection{Tensor low-rank decomposition and approximation} 
    
    Let
    $
    \X = \X_1\times \cdots \times \X_N
    $
    be a Cartesian product, and let
    $
    F : \X \to \R
    $
    be a multivariate function. 
    A usual tensor approximation of rank $r$ 
    seeks a representation of the form
    \begin{equation}\label{eq:classical_low_rank}
    	F(x_1,\dots,x_N)\approx \sum_{q=1}^r \prod_{i=1}^N f_{q,i}(x_i),
    \end{equation}
    where each factor depends only on one variable. 
    After discretization on a tensor-product grid, \eqref{eq:classical_low_rank} becomes a tensor low-rank decomposition. 
    This viewpoint underlies canonical, Tucker, tensor-train, and hierarchical tensor formats, whose purpose is to replace a high-dimensional object by a structured representation with moderate complexity.

   Usual tensor low-rank  approximation of the solution of HJ equations have been proposed in \cite{zbMATH07364328,zbMATH07924430}.
    Here, tensor low-rank approximations will be understood in the tropical sense, leveraging the tropical (max-plus) linearity of the Lax-Oleinik semigroup. 
    
	\section{TROPICAL TENSOR LOW-RANK APPROXIMATION AND OPTIMAL CONTROL OF $N$-BODY SYSTEM}\label{sec-lowrank}
	
	We now introduce the tropical tensor low-rank viewpoint and relate it to the optimal control of $N$-body systems. Whereas we introduced above the notations $\oplus$ and $\odot$ to highlight the linearity property of the Lax-Oleinik semigroups, we will now use the more familiar symbols $\vee$ for maximization or pointwise maximization (instead of $\oplus$) and $+$ (instead of $\odot$).
	
	\subsection{Max-plus separable and tensor low-rank approximation}
	
	Let $\X=\X_1\times \cdots \times \X_n \subset (\R^d)^n$ be a compact product set. 
	A function $G:\X\to \R_{\max}$ is said to have \emph{tropical tensor rank one} if it can be written as
	\begin{equation*}\label{eq:trop_rank_one}
		G(x_1,\dots,x_n)=\sum_{i=1}^n g_i(x_i),
	\end{equation*}
	for some functions $g_i:\X_i\to \R_{\max}$. 
	More generally, a tropical tensor rank-$r$ approximation of $F:\X\to \R_{\max}$ is an expression of the form
	\begin{equation}\label{eq:trop_rank_r}
		F(x)\approx \bigvee_{k=1}^r \tilde F^k(x),
		\
		\tilde F^k(x):=\sum_{i=1}^n \tilde F_i^k(x_i) \ ,
	\end{equation}
	where $\vee$ denotes the pointwise maximum or supremum. 
	In this paper, we mainly consider \emph{lower} approximations,
	\begin{equation*}\label{eq:lower_approx}
		\bigvee_{k=1}^r \tilde F^k(x)\le F(x),
		\qquad x\in \X\enspace.
	\end{equation*}
	
	Finite suprema of tropical tensor rank-one functions form a rich approximation class on compact product sets. 
	
	\begin{proposition}\label{prop:dense_tropical_separable}
		Let $\X=\X_1\times\cdots\times \X_n\subset (\R^d)^n$ be compact, and let $F$ be a continuous function $X\to \R$. 
		Then, for every $\varepsilon>0$, there exist $r\ge 1$ and continuous functions $\tilde F_i^k:\X_i\to \R$ such that
		\begin{equation*}\label{eq:dense_tropical_separable}
			F(x)-\varepsilon
			\le
			\bigvee_{k=1}^r \sum_{i=1}^n \tilde F_i^k(x_i)
			\le
			F(x),
			\ \forall x\in \X.
		\end{equation*}
	\end{proposition}
        The proof is elementary. In fact, the functions $\tilde F_i^k$ can be chosen to be quadratic, the result being reminiscent of the Moreau-Yoshida regularization. 
	In some situations, the purely separable form \eqref{eq:trop_rank_r} is too restrictive. 
	One may take generalized tensor low-rank classes
	\begin{equation}\label{eq:generalized_tropical_rank_r}
		F(x)\approx \bigvee_{k=1}^r \tilde G^k(x),
		\ 
		\tilde G^k(x):=\sum_{i=1}^n \tilde G_i^k(A_i^k x) \ ,
	\end{equation}
	where $A_i^k$ are low-dimensional linear projections. 
	The basic separable form \eqref{eq:trop_rank_r} is recovered by taking $A_i^k$ equal to the canonical block projections.
	
	The representation \eqref{eq:generalized_tropical_rank_r} should be viewed as an intermediate model class between fully separable tropical tensor rank-one terms and a  general high-dimensional form. 
	It allows each factor to depend on a projected variable, which can encode part of the interaction structure while keeping the complexity controlled. 
	In the present paper we mainly work with the purely separable quadratic supporting functions generated by the algorithm; the same trajectory-based philosophy could be combined with projected tensor low-rank classes in future work.
	
	\subsection{Optimal control formulation of a $N$-body system}
	
	Let us consider $N$ interacting subsystems, with joint state 
	\[
	\xi(s) = (\xi_1(s),\dots,\xi_N(s)) \in (\R^d)^N \ .
	\]
	For each particle $i$, let $V_i:\R^d\to\R$ be an individual potential and
	$
	T_i(\dot{\xi}_i):=\frac12\dot{\xi}_i^\top M_i\dot{\xi}_i
	$
	a quadratic kinetic term, where $M_i$ is a symmetric positive definite matrix.
	Let
	\begin{equation*}\label{interaction}
		W:(\R^d)^N\to \R
	\end{equation*}
	be an interaction potential. 
	For a horizon $t>0$, an initial state $x$, and a terminal cost $\Phi_t$, we consider the Bolza functional
	\begin{equation*}\label{action_function}
		\begin{aligned}
		\mathcal J_t(\xi(\cdot))
		:=
		\int_0^t
		&\Big(
                  \sum_{i=1}^N \big( V_i(\xi_i(s)) + T_i(\dot{\xi}_i(s)) \big)
                  \\ 
		&\qquad \qquad + W(\xi(s))
		\Big) ds
		+ \Phi_t(\xi(t)) \ .
		\end{aligned}
	\end{equation*}
	The associated value function is
	\begin{equation*}\label{eq:manybody_value}
		\mathcal V(x,t)
		:=
		\inf_{\xi \in W^{1,\infty}([0,t];(\R^d)^N),\ \xi(0)=x}
		\mathcal J_t(\xi(\cdot)).
	\end{equation*}
	
	Under standard coercivity assumptions, this variational problem is well posed.
	
	\begin{proposition}\label{prop:existence_many_body}
		Suppose that the matrices $M_i$ are symmetric positive definite, the functions $V_i$, $W$, and $\Phi_t$ are continuous and bounded from below, and the total potential is coercive at infinity. 
		Then, for every $x\in(\R^d)^N$ and $t>0$, the functional $\mathcal J_t$ admits a minimizer in
		$
		\left\{
		\xi \in W^{1,\infty}([0,t];(\R^d)^N)  \mid \xi(0)=x
		\right\}.
		$ 
	\end{proposition}
	
	Introducing the control $u=\dot{\xi}$, and fixing the horizon $T$, we can rewrite the problem as a deterministic optimal control problem with $v(x,0)=-\mathcal V(x,T)$ and 
	\begin{equation*}\label{Nbody_to_control}
		\left\{
		\begin{aligned}
			&x = \xi = (\xi_1,\dots,\xi_N), \qquad u=\dot{\xi},\\
			&\ell(x,u)
			=
			-
			\left(
			\sum_{i=1}^N \Big( V_i(x_i) + \frac12 u_i^\top M_i u_i \Big)
			+ W(x)
			\right),\\
			&f(x,u)=u,\\
			&\phi = -\Phi_T \enspace . 
		\end{aligned}
		\right.
	\end{equation*}
	Hence the $N$-body value function coincides with the solution of the associated HJB PDE~\eqref{eq:HJB},  up to a change of sign. %
	
	\subsection{Max-plus separability and tensor low-rank approximation in $N$-body systems}
	
	The running cost naturally splits into a local part and an interaction part, so 
	$
	\ell(x,u)=\ell_{\rm loc}(x,u)-W(x),
	$
	where
	\begin{equation*}\label{eq:local_lagrangian}
		\ell_{\rm loc}(x,u)
		:=
		-\sum_{i=1}^N
		\left(
		V_i(x_i)+\frac12 u_i^\top M_i u_i
		\right).
	\end{equation*}
	The local term is additively separable across the particles, whereas $W$ creates the high-dimensional coupling.
	
	\begin{proposition}\label{prop:exact_decoupled}
		Suppose that $W\equiv 0$ or that $-W$ has tropical tensor rank one, and that the terminal reward $\phi$ has tropical tensor rank one, 
		then the value function is separable, that is
		$
		\mathcal V(x,t)=\sum_{i=1}^N \mathcal V_i(x_i,t),
		$
		where $\mathcal V_i$ is the value function of the $i$-th subsystem. 
		More generally, finite maxima of separable terminal rewards are preserved by the dynamic programming.
	\end{proposition}

	
	
	The elementary \Cref{prop:exact_decoupled} shows that tropical tensor low-rank, in particular rank-one, representations are exact in the fully decoupled case. In general, however, the set of tropical tensors with rank $r$ is not preserved by repeated time propagation. If $-W$ is approximated by a supremum of $R$ separable terms and the terminal reward has tropical tensor rank $r_0$, then after $N$ time steps a direct dynamic programming scheme typically produces a supremum of at most $r_0R^N$ separable terms. Hence direct max-plus propagation suffers from exponential rank growth with the time horizon. 
	This motivates the trajectory-based update strategy of the next section.
	
	\section{A trajectory-based iterative tropical tensor low-rank approximation}\label{section-algo}
	
	\subsection{Tropical tensor low-rank approximation of semiconvex value functions}
	
	We first discretize the horizon by a time step $h>0$, 
	\[
	t_k := kh, \ k=0,1,\dots,K \text{ and } Kh=T.
	\]
	Let $v_k^h$ be the discrete-time value functions, defined by the Bellman equation
\begin{subequations}
	\begin{equation}\label{eq:bellman_equation}
	v_k^h=T_h[v_{k+1}^h],\; k=0,\dots,K-1, \qquad v_K^h=\phi\enspace ,
	\end{equation}
where $T_h$ is the 
discrete Bellman operator acting on the set of bounded functions $\psi:X\to\R$:
  \begin{gather}
    (T_h\psi)(x):=\sup_{u\in U}\Bigl\{h\,\ell(x,u)+\psi(F_h(x,u))\Bigr\}\ ,
    \label{eq:bellman_short}\\
        \text{with} \; F_h(x,u)=x+h f(x,u), \; x\in X,\; u\in U\enspace .
        \end{gather}
\end{subequations}

	The algorithm relies on the existence of global lower supporting quadratic functions, so that we make the following assumption. 
	
	\begin{assumption}\label{ass:traj_algo}
		We assume that:
		\begin{enumerate}
			\item[(i)] the control set $U\subset \R^m$ is nonempty and compact;
			\item[(ii)] the one-step dynamics is affine, i.e., 
			\[
			F_h(x,u)=A_hx+B_hu+b_h;
			\]
			\item[(iii)] the terminal reward $\phi$ is continuous and $c_K$-semiconvex;
			\item[(iv)] there exists $\gamma_h\ge 0$ such that, for every $u\in U$, the map $x\mapsto h\,\ell(x,u)$ is continuous and $\gamma_h$-semiconvex w.r.t.\ $x$,  uniformly in $u$.
		\end{enumerate}
	\end{assumption}
	
	Under this assumption, semiconvexity is propagated by the Bellman operator $T_h$ given in \eqref{eq:bellman_short}. 
	
	\begin{proposition}\label{prop:bellman_preserves_semiconvexity}
	  Suppose \Cref{ass:traj_algo} holds. 
		Define recursively
		\begin{equation*}\label{eq:curvature_recursion}
			c_k := \gamma_h + \|A_h\|^2 c_{k+1},
			\ k=K-1,\dots,0.
		\end{equation*}
		Then, $v_k^h$ is $c_k$-semiconvex, for every $k$.
	\end{proposition}
\begin{proof}
		The proof is by backward induction. 
		If $v_{k+1}^h$ is $c_{k+1}$-semiconvex, then for each fixed control $u$,		\[
		x\mapsto h\,\ell(x,u)+v_{k+1}^h(F_h(x,u))
		\]
		is $\gamma_h+\|A_h\|^2c_{k+1}$-semiconvex, because $x\mapsto h\,\ell(x,u)$ is uniformly semiconvex and $F_h(\cdot,u)$ is affine. 
		Taking the supremum over $u\in U$ preserves semiconvexity.
	\end{proof}
	
	The next proposition yields global lower concave quadratic supporting functions.
	
	\begin{proposition}\label{prop:global_semiconvex_support}
		Let $f:\R^n\to \R$ be a $c$-semiconvex function. 
		Then, for every $a\in \R^n$, there exists $p\in \R^n$ such that
		\begin{equation*}\label{eq:generic_support}
			w(x):=f(a)+p\cdot (x-a)-\frac c2\|x-a\|^2
		\end{equation*}
		satisfies
		\(
		w(x)\le f(x), \ \forall x\in \R^n,
		\ \text{and}\
		w(a)=f(a) \ .
		\)
	\end{proposition}
	
	We therefore approximate the discrete value functions by finite suprema of concave quadratic supporting functions,
	\begin{equation}\label{eq:maxplus_quad_approx}
		\widehat v_k^{h,m}(x)
		=
		\bigvee_{j=1}^{r_k^{(m)}} w_{k,j}^{h,m}(x) \ ,
	\end{equation}
	where $m\ge 0$ is the outer iteration index and each basis function has the form
	\begin{equation}\label{eq:quadratic_basis_section}
		w_{k,j}^{h,m}(x)
		=
		\beta_{k,j}^{h,m}
		+
		p_{k,j}^{h,m}\cdot (x-a_{k,j}^{h,m})
		-
		\frac{c_k}{2}\|x-a_{k,j}^{h,m}\|^2 .
	\end{equation}
	In the $N$-body setting $x=(x_1,\dots,x_N)\in(\R^d)^N$, the quadratic term is additively separable across the particles, hence each basis function is a tropical tensor of rank one.
	
	\subsection{The iterative procedure} 
	
	The method alternates a backward enrichment step and a forward trajectory update. 
	It starts from a prescribed initial point $x_0$, an admissible reference trajectory
	\[
	x^{(0)}=(x_0^{(0)},\dots,x_K^{(0)}),
	\ x_0^{(0)}=x_0,
	\]
	and initial lower bounds
	\[
	\widehat v_k^{h,0}\le v_k^h,
	\ k=0,\dots,K.
	\]
	A simple choice is to take a nominal trajectory generated by a constant or zero control and to initialize $\widehat v_k^{h,0}$ by a certified constant lower bound.
	
	The role of the reference trajectory is only to indicate where the next supporting functions should be generated, so it does not need to be close to an optimal trajectory. 
	Likewise, the initial lower bounds may be very coarse, provided that they are certified (one can take $-\infty$).
	This is important in practice, because it allows the algorithm to start from inexpensive surrogates and then improve them monotonically.
	
        We then construct, at each outer iteration $m\ge 1$,
        an approximation of the value function of the form 
        \eqref{eq:maxplus_quad_approx},         as follows.
        To simplify the presentation, we assume that $\phi=\widehat v_K^{h,0}$ is already in the form \eqref{eq:maxplus_quad_approx}.
        For $k=K-1,\dots,0$, we first compute a greedy backward control
	\begin{equation}\label{eq:local_greedy_control_backward}
		\begin{aligned}
		u_k^{(m)}(w)\in \argmax_{u\in U}
		&\left\{
		h\,\ell(x_k^{(m-1)},u) \right. \\ 
		& \quad  \left.+
		w(F_h(x_k^{(m-1)},u))
		\right\}.
		\end{aligned}
	\end{equation}
        for each function $w=w_{k+1,j}^{h,m-1}$, with $j\leq r_{k+1}^{(m-1)}$, then deduce
     $u_k^{(m)}=u_k^{(m)}(\widehat v_{k+1}^{h,m-1})$.
	We then select an active basis function $w_{k+1,j}^{h,m-1}$ of $\widehat v_{k+1}^{h,m-1}$ at the successor state $x_{k+1}=F_h(x_k^{(m-1)},u_k^{(m)})$ meaning
        $w_{k+1,j}^{h,m-1}(x_{k+1})=\widehat v_{k+1}^{h,m-1}(x_{k+1})$ 
        and setting $\iota_k^{(m)}=j$, we define the one-step update
	\begin{equation*}\label{eq:one_step_profile}
		G_k^{(m)}(x)
		:=
		h\,\ell(x,u_k^{(m)})
		+
		w_{k+1,\iota_k^{(m)}}^{h,m-1}(F_h(x,u_k^{(m)})).
	\end{equation*}
	By \Cref{prop:global_semiconvex_support}, there exists a concave quadratic supporting function
	\begin{equation}\label{eq:new_support_algo}
		\begin{aligned}
		&w(x)
		=
		G_k^{(m)}(x_k^{(m-1)}) \\
		&\quad +
		p_k^{h,m}\cdot(x-x_k^{(m-1)}) 
		-
		\frac{c_k}{2}\|x-x_k^{(m-1)}\|^2
		\end{aligned}
	\end{equation}
	such that $w\le G_k^{(m)}$ and equality holds at $x_k^{(m-1)}$. 
	If $w$ does not belong to the set of $w_{k,j}^{h,m-1}$, with $j\leq r_{k}^{(m-1)}$, 
        we set $r_k^{(m)}= r_k^{(m-1)}+1$, and
        $w_{k,r_k^{(m)}}^{h,m}=w$ otherwise $r_k^{(m)}= r_k^{(m-1)}$. Then, 
        the approximation of the value function at time $k$ is updated by
	\begin{equation}\label{eq:enrichment_update}
		\widehat v_k^{h,m}
		=
		\widehat v_k^{h,m-1}\vee w.
	\end{equation}
	
	Once the backward update is complete, we recompute the reference trajectory by a forward greedy pass:
	\begin{equation}\label{eq:local_greedy_control_forward}
		u_k^{(m)}
		\in
		\argmax_{u\in U}
		\left\{
		h\,\ell(x_k^{(m)},u)
		+
		\widehat v_{k+1}^{h,m}\big(F_h(x_k^{(m)},u)\big)
		\right\},
	\end{equation}
	followed by
	\[
	x_{k+1}^{(m)}=F_h(x_k^{(m)},u_k^{(m)}),
	\qquad k=0,\dots,K-1.
	\]

	\begin{algorithm}[t]
		\caption{Trajectory-based iterative tropical tensor low-rank approximation}
		\label{alg:traj_tropical}
		\begin{algorithmic}[1]
		  \Require Initial point $x_0$, initial trajectory $x^{(0)}$, initial lower bounds $(\widehat v_k^{h,0})_{k=0}^K$ s.t.\ $\widehat v_K^{h,0}=\phi$, number of iterations $M$
			\For{$m=1,\dots,M$}
			\For{$k=K-1,\dots,0$}
			\State Compute $u_k^{(m)}$ by \eqref{eq:local_greedy_control_backward}
			\State Select an active basis function of $\widehat v_{k+1}^{h,m-1}$
			\State Build a quadratic supporting function $w$ of $G_k^{(m)}$
			\State Update $\widehat v_k^{h,m}\gets \widehat v_k^{h,m-1}\vee w$
			\EndFor
			\State Set $x_0^{(m)}\gets x_0$
			\For{$k=0,\dots,K-1$}
			\State Compute $u_k^{(m)}$ by \eqref{eq:local_greedy_control_forward}
			\State Propagate $x_{k+1}^{(m)}\gets F_h(x_k^{(m)},u_k^{(m)})$
			\EndFor
			\EndFor
		\end{algorithmic}
	\end{algorithm}
		The resulting scheme is summarized in~\Cref{alg:traj_tropical}. 
	The basic structural properties of this algorithm are the following.
	
	\begin{proposition}\label{prop:monotone_lower_bounds}
		Suppose that \Cref{ass:traj_algo} holds and that
		\[
		\widehat v_k^{h,0}\le v_k^h,
		\ k=0,\dots,K.
		\]
		Then, for every  $m\ge 1$ and every $k=0,\dots,K$,
		\begin{equation*}\label{eq:monotone_bounds}
			\widehat v_k^{h,m-1}\le \widehat v_k^{h,m}\le v_k^h .
		\end{equation*}
		In particular, the sequence $\big(\widehat v_k^{h,m}\big)_{m\ge 0}$ is monotone nondecreasing and remains a family of lower bounds for $v_k^h$. 
		Moreover,
		\begin{equation*}\label{rank_increase}
			r_k^{(m)}\le r_k^{(0)}+m,
			\ k=0,\dots,K,
		\end{equation*}
		so the tropical tensor rank grows at most linearly with the number of outer iterations.
	\end{proposition} 
	
	\begin{proof}
		Each selected basis function of $\widehat v_{k+1}^{h,m-1}$ is bounded above by $v_{k+1}^h$, hence the one-step profile $G_k^{(m)}$ is bounded above by the Bellman update corresponding to the specific control $u_k^{(m)}$, and therefore by $v_k^h$. 
		Since the new supporting function satisfies $w_k^{h,m}\le G_k^{(m)}$, we obtain $w_k^{h,m}\le v_k^h$. 
		The update \eqref{eq:enrichment_update} then immediately yields
		\[
		\widehat v_k^{h,m-1}\le \widehat v_k^{h,m}\le v_k^h.
		\]
		The rank estimate follows because each backward sweep adds at most one basis function per time step.
	\end{proof}
	
	\Cref{prop:monotone_lower_bounds} is the key reason for using lower supporting functions, that is the approximation improves monotonically while preserving a certified bound. 
	This also avoids the uncontrolled rank explosion of naive propagation.

	\begin{remark}
		When $\ell$ and the active basis function are differentiable, the vector $p_k^{h,m}$ in \eqref{eq:new_support_algo} can be chosen as a gradient of the one-step update $G_k^{(m)}$ at the touching point. 
		For affine dynamics $F_h(x,u)=A_hx+B_hu+b_h$, this yields an explicit expression obtained by the chain rule, which makes the construction of supporting functions inexpensive once the active basis function is known.
	\end{remark}
	
	\subsection{Tropical tensor low-rank approximation in the optimal control of $N$-body systems}
	
	We now return to the optimal control of $N$-body systems.
	Since the state variable has the product form
	\[
	x=(x_1,\dots,x_N)\in (\R^d)^N \ ,
	\]
	the quadratic basis function \eqref{eq:quadratic_basis_section} decomposes as
	\[
	\begin{aligned}
	w_{k,j}^{h,m}(x)
	=
	\beta_{k,j}^{h,m}
	+
	\sum_{i=1}^N
	\big( &
	p_{k,j,i}^{h,m}\cdot (x_i-a_{k,j,i}^{h,m}) \\
	&-
	\frac{c_k}{2}\|x_i-a_{k,j,i}^{h,m}\|^2
	\big) \ .
	\end{aligned}
	\]
    So every generated supporting function is additively separable and therefore has tropical tensor rank one. 
	It follows that
	\[
	\widehat v_k^{h,m}(x)
	=
	\bigvee_{j=1}^{r_k^{(m)}} w_{k,j}^{h,m}(x)
	\]
	is a tropical tensor low-rank approximation of the value function.  
	This point of view also leads to a favorable complexity estimate. 
	For a state dimension \(n=Nd\), storing or evaluating one basis function requires \(O(Nd)\) operations. 
	Hence, if \(\widehat v_k^{h,m}\) contains \(r_k^{(m)}\) basis functions, one evaluation of \(\widehat v_k^{h,m}\) costs \(O(r_k^{(m)}Nd)\).

\subsection{A Glimpse of the Convergence Result}
	
	We briefly discuss the convergence of the trajectory-based tropical approximation. 
	The key point is that the algorithm generates a monotone sequence of lower bounds of the true value function, and that, under additional regularity and small-time assumptions, the limiting approximation is exact along every limiting trajectory produced by the scheme.
	
	
	We assume throughout this section that $X$ is compact and convex, $U$ is compact, $F_h(X\times U)\subset X$, the maps $F_h$ and $\ell$ are continuous and Lipschitz continuous in the state variable, and that the initial approximation $\widehat v^{h,0}$ is Lipschitz continuous and satisfies the Bellman subsolution property: 
	\[
	\widehat v_k^{h,0}\le T_h[\widehat v_{k+1}^{h,0}],\quad k=0,\dots,K-1\enspace,
	\]
        where $T_h$ is the discrete Bellman operator given in \eqref{eq:bellman_short}.
	Under these assumptions, the monotonicity of the Bellman operator $T_h$ and the construction of the algorithm imply that the sequence $(\widehat v_k^{h,m})_{m\ge0}$ remains monotone nondecreasing and uniformly Lipschitz continuous.
	
	\begin{theorem}\label{thm:short_convergence}
		For every $k=0,\dots,K$, the sequence $(\widehat v_k^{h,m})_{m\ge0}$ converges uniformly on $X$ to a Lipschitz continuous function $v_k^{h,*}$ satisfying
		$
		\widehat v_k^{h,m}\le v_k^{h,*}\le v_k^h  \text{ on }X,
		$ 
		and
		$
		v_k^{h,*}\le T_h[v_{k+1}^{h,*}], k=0,\dots,K-1.
		$ 
		Moreover, let $(x^{(m)})_{m\ge0}$ be the sequence of forward trajectories $x^{(m)}=(x^{(m)}_k)_{k=0,\ldots, K}$ generated by the algorithm. For every convergent subsequence
		\(
		x^{(m_j)}\to x^*:=(x_k^*)_{k=0,\ldots, K}\in X^{K+1},
		\)
		there exists a subsequence and
                controls $u_k^*\in U$, $k=0,\dots,K-1$,
                such that $x_{k+1}^*=F_h(x_k^*,u_k^*)$,
		\[
		u_k^*\in \argmax_{u\in U}\Bigl\{h\,\ell(x_k^*,u)+v_{k+1}^{h,*}(F_h(x_k^*,u))\Bigr\} \ ,
		\]
		and 
		\begin{equation}\label{eq:bellman_limit_traj_short}
			v_k^{h,*}(x_k^*)
			=
			h\,\ell(x_k^*,u_k^*)+v_{k+1}^{h,*}(x_{k+1}^*) \ .
		\end{equation}
	\end{theorem}
	
	\begin{skproof}
          The proof is similar to the ones showing the  convergence of SDDP algorithm, see~\cite{akian2025stochastic}.
		The monotonicity of the algorithm yields $\widehat v_k^{h,m}\le \widehat v_k^{h,m+1}\le v_k^h$, hence the pointwise monotone convergence of $\widehat v_k^{h,m}$. Uniform Lipschitz bounds imply compactness, and the monotone limit is therefore uniform on $X$. Passing to the limit in the Bellman subsolution inequality gives
		$
		v_k^{h,*}\le T_h[v_{k+1}^{h,*}].
		$
		Finally, compactness of $U$ and stability of maximizers yield the convergence of a subsequence of controls, while the contact relation enforced by the backward step passes to the limit and gives \eqref{eq:bellman_limit_traj_short}.
	\end{skproof}

	To identify this limit along a trajectory, one may use the following result.
        \begin{theorem}\label{thm:short_exactness}
          Assume in addition that, 
          for each $k=0,\dots,K$, the function $v_k^h$ is semiconcave and that the associated one-step Bellman maximization is uniformly strongly concave in the control variable. Then, 
          the sequence $(x^{(m)})_{m\ge0}$ converges towards the unique optimal trajectory $x^*=(x_k^*)_{0\le k\le K}$ and we have 
		\[
		v_k^{h,*}(x_k^*)=v_k^h(x_k^*),\qquad k=0,\dots,K.
		\]
	\end{theorem}
	
	\begin{skproof}
          Let  $x^*$ be a limit trajectory  as in \Cref{thm:short_convergence}.
       Let us show the above equality by backward induction on $k=0,\ldots , K-1$. Consider the two functions
          $q(u)=h\,\ell(x_k^*,u)+v_{k+1}^h(F_h(x_k^*,u))$ and
          $q^*(u)=h\,\ell(x_k^*,u)+v_{k+1}^{h,*}(F_h(x_k^*,u))$.
       By assumption, $q$ is semiconcave and  $q^*$ is semiconvex.
       Moreover $q^*\leq q$ 	  since $v_{k+1}^{h,*}\le v_{k+1}^h$.
          If $v_{k+1}^{h,*}$ and $v_{k+1}^h$ coincide at  $x_{k+1}^*$, we have that
          $q(u_k^*)=q^*(u_k^*)$. The contact between a semiconvex function and a semiconcave one implies equality of their gradients at $u_k^*$, leading to the
          optimality of $u_k^*$ for Bellman maximization problem for $v_{k+1}^h$ at $x_k^*$, and the equality of $v_{k}^{h,*}$ and $v_{k}^h$ at point $x_{k}^*$.
          By strong concavity in the control variable, the maximizer is unique, hence the limit control $u_k^*$ must coincide with the unique optimal feedback.
		A backward induction concludes the proof.
	\end{skproof}
        	
	\begin{remark}
          The discrete control problem considered here may be viewed as a finite-dimensional multi-stage optimization problem, and the previous result shows that the limiting trajectory generated by the algorithm is globally optimal for this problem. 
A counter example of \cite{akian2025stochastic} shows that
the conclusion of  \Cref{thm:short_exactness} is false in general.
This is false in particular for one time step, and some final reward equal to the supremum of two quadratic functions (which is then semiconvex and can be represented  by a tropical tensor of rank 2 but is not semiconcave).
	\end{remark}
        In order to guarantee the semiconcavity of the exact discrete value functions $v_k^h$, as assumed in the previous result, one may strengthen the standing assumptions, for example, by assuming that the terminal reward is semiconcave, that the running reward is uniformly semiconcave in the state variable and strongly concave in the control variable, and that the associated one-step Bellman maximization problem is uniformly strongly concave. 
        In particular, a $C^{1,1}$ regularity of the reward function is sufficient. 

	\section{NUMERICAL RESULTS} \label{numerics_nbody}

	In this section, we apply our algorithm to 
        an $N$-body system described in Section~\eqref{section-algo}.  
	\subsection{Problem setup}
	
We consider a planar swarm with $N$ agents over a horizon $T=20$. 
Each agent evolves in dimension $2$, so the state and control are $x,u\in\R^{2N}$.  
	The dynamics, and the running and terminal rewards of the associated deterministic optimal control problem (such that $v(x,0)=-\mathcal V(x,T)$) are
	\begin{equation*}
		\left\{
		\begin{aligned}
                  &f(x,u)=u\ , \\
			&\ell(x,u) = -\Big(\tfrac12 k_{\rm trap}\|x\|^2 + W_\varepsilon(x) + \tfrac12 u^\top R\,u\Big) \ , \\ 
			&\Phi_T(x) \;=\; -\tfrac12 k_T \|x\|^2 \ ,
		\end{aligned}
		\right.
	\end{equation*}
	with $
	k_{\rm trap}=0.35, k_T=5.0, R=0.5\,I_{2N}$. 
	The interaction is a softened Coulomb potential
	\begin{align*}
		W_\varepsilon(x) = \sum_{1\le i<j\le N}\frac{\kappa}{\sqrt{\|x_i-x_j\|^2+\varepsilon^2}} \ ,
	\end{align*}
	where $\varepsilon=0.2, \kappa = 1.5$.
        We apply the time discretization \eqref{eq:bellman_equation},
         with uniform step $h=\Delta t=0.01$ ($K=2000$). 
	We run $500$ outer (rank) iterations. 
	A no-Coulomb LQR solution is computed as a simple baseline lower bound for $\mathcal V(x,T)$.
	
	Because each backward sweep adds one supporting function at every time step, the total number of supporting functions grows linearly with the number of iterations. 
	In practice, many supporting functions are never active. 
	We therefore use an optional pruning 
        step (see, for instance, \cite{gaubert2011curse,gaubert2014bundle,Qu2014}) that probes the current max-plus expansion at random points and retains only the supporting functions that are active at least once. 
	This considerably reduces runtime, while preserving the same qualitative behavior.

	\subsection{Numerical results: symmetric initial states.}\label{sec-app1}

	We first consider equally spaced initial positions on a circle of radius $10$. 
        \Cref{fig:application1_traj} shows the trajectories obtained after $500$ iterations for $N=5$ and $N=10$. 
	In both cases, the agents rapidly converge 
          to a nearly circular configuration and then remain close to a steady annulus for most of the horizon before moving inward near the terminal time. 
	This phenomenon is consistent with the \emph{turnpike property}~\cite{trelat2015turnpike,avskovic2024linear,trelat2025turnpike} in optimal control, 
	which asserts that for long horizons, optimal trajectories spend most of their duration near a steady-state configuration minimizing the instantaneous Lagrangian. 
	
	The corresponding equilibrium radius is determined by the balance between the inward quadratic trap and the outward Coulomb repulsion. 
	If the particles are approximately evenly distributed on a circle of radius $r$, a simple scaling argument gives
	\begin{equation*}\label{eq:turnpike_radius}
		r^* = O\Big(\tfrac{\kappa\,N^2}{k_{\mathrm{trap}}}\Big)^{1/3},
	\end{equation*}
	which is consistent with the fact that the annulus becomes larger as $N$ increases. 
	Near terminal time, the terminal cost dominates and produces the final coordinated contraction towards the origin. 

	\begin{figure}[htb]
		\centering
		\subfloat[Trajectories obtained at iteration $500$ for a $5$ body system with symmetric initial states spaced on a circle of radius $10$. All agents move symmetrically towards the origin.\label{fig:cas1_1}]{
			\includegraphics[width=0.44\linewidth]{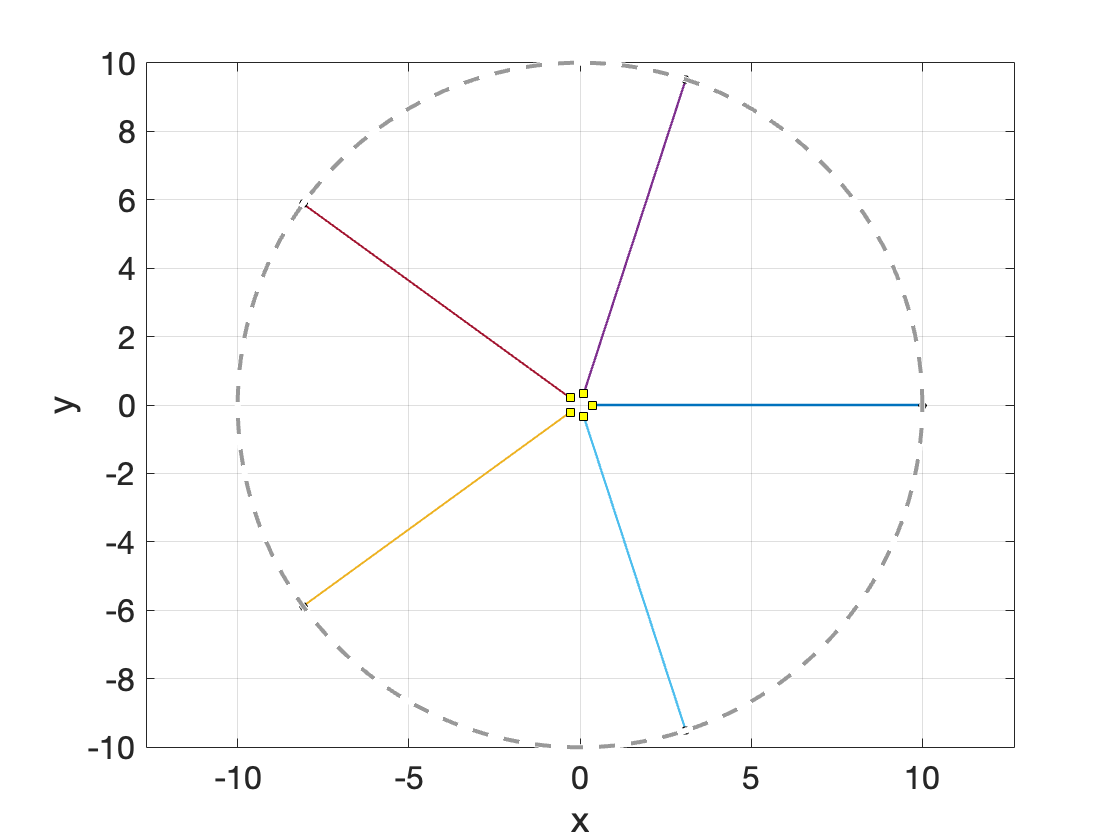}
		}
		\hfill
		\subfloat[The evolution of the agents’ radii $|x_i|$ exhibits a \emph{turnpike} behavior. Agents remain on a nearly constant-radius circle for most of the horizon, then move inward sharply near the end.\label{fig:case1_2}]{
			\includegraphics[width=0.44\linewidth]{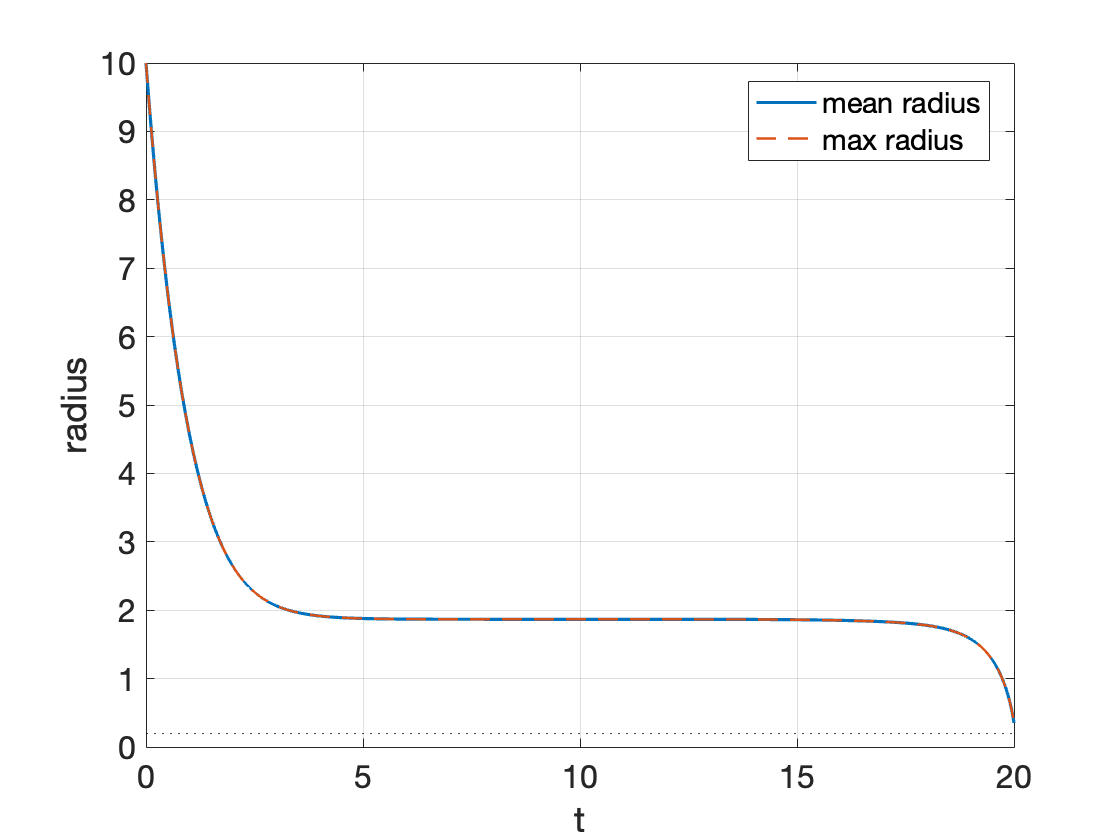}
		}
		
		
		\subfloat[Trajectories obtained at iteration $500$ for a $10$ body system with symmetric initial states spaced on a circle of radius $10$.\label{fig:case2_1}]{
			\includegraphics[width=0.44\linewidth,trim=0 0 0 2cm, clip]{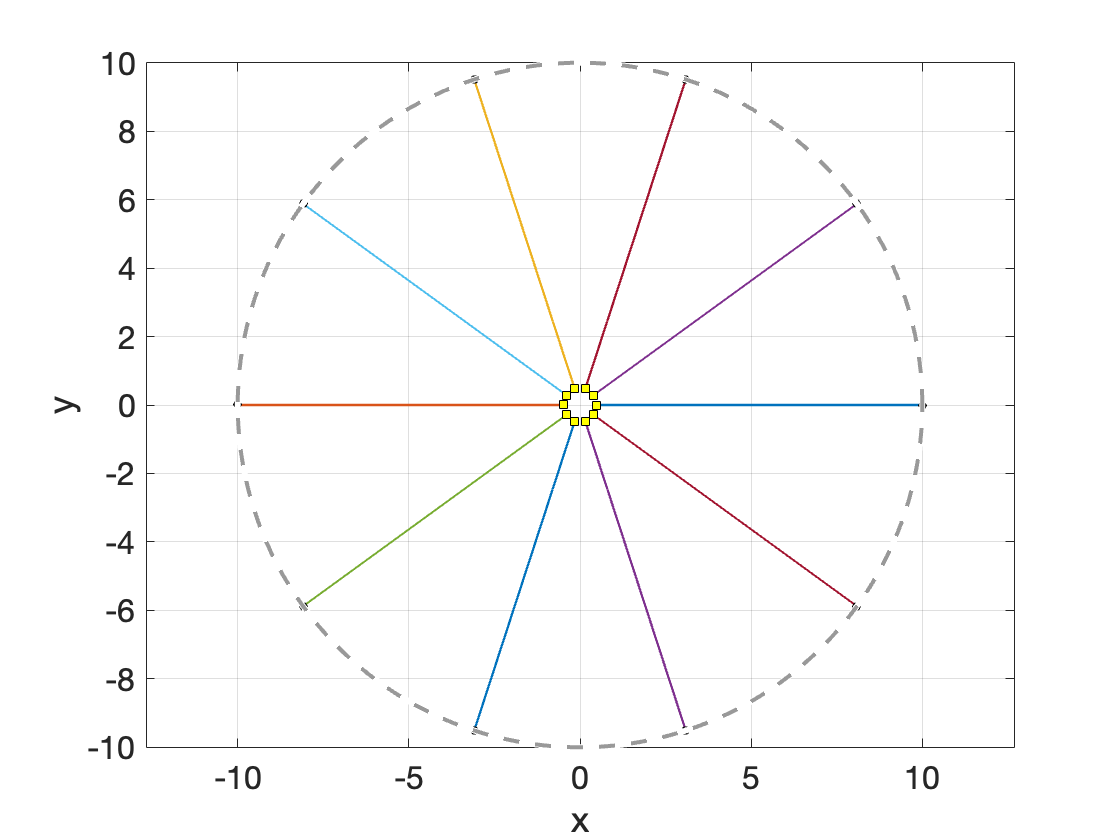}
		}
		\hfill
		\subfloat[Turnpike behaviour is also observed for the $10$ body system, with larger radii than for the $5$ body system.\label{fig:case2_2}]{
			\includegraphics[width=0.44\linewidth]{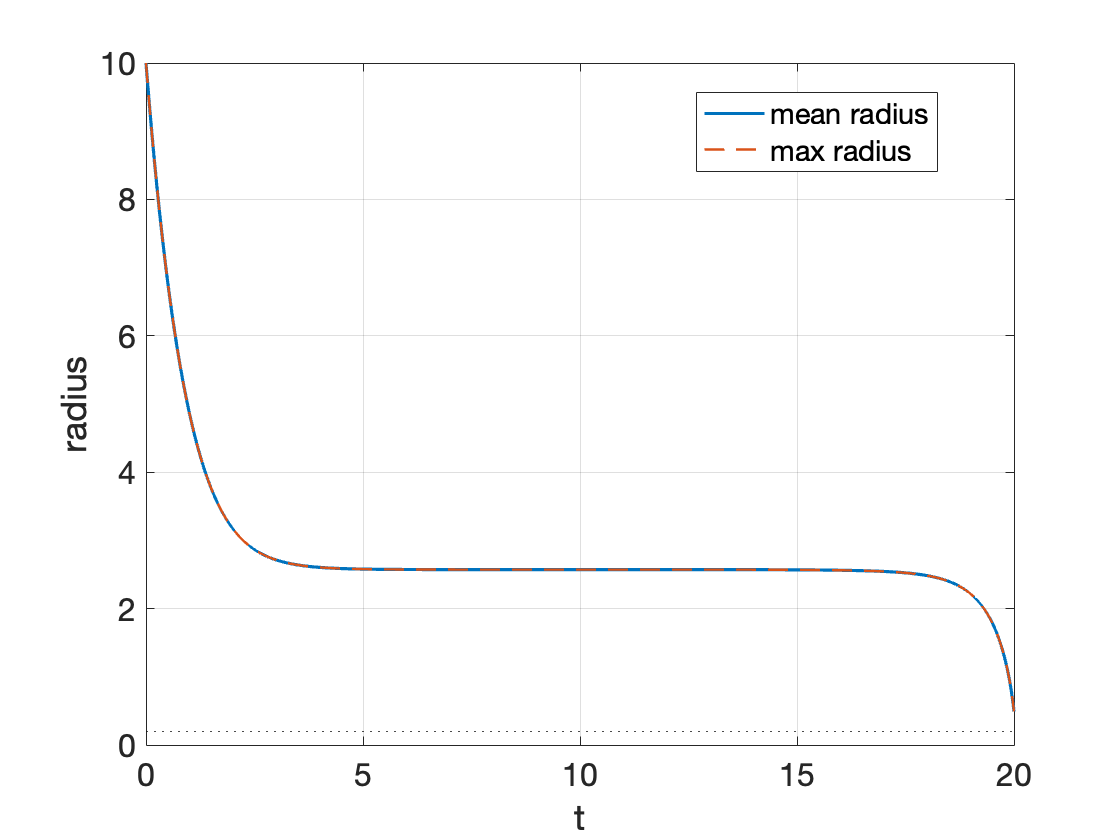}
		}
		\caption{Optimal trajectories for equally spaced initial positions on a circle of radius $10$.}
		\label{fig:application1_traj}
	\end{figure}

	\begin{figure}[htb]
		\centering
		\subfloat[Value of $\mathcal V(x_0,T)$ w.r.t.\ number of iterations, for a $5$ body system.\label{fig:cas1_value1}]{
			\includegraphics[width=0.44\linewidth]{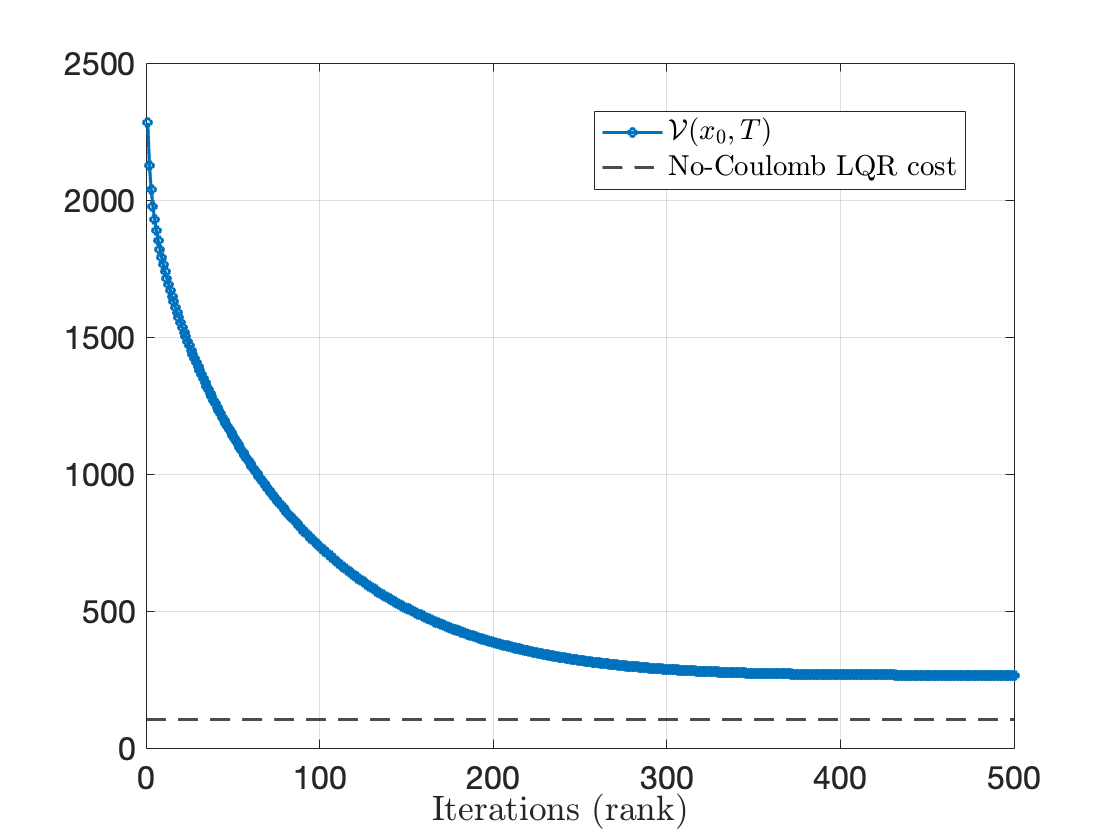}
		}
		\hfill
		\subfloat[Value of $\mathcal V(x_0,T)$ w.r.t.\ number of iterations, for the $10$ body system.\label{fig:case1_value2}]{
			\includegraphics[width=0.44\linewidth]{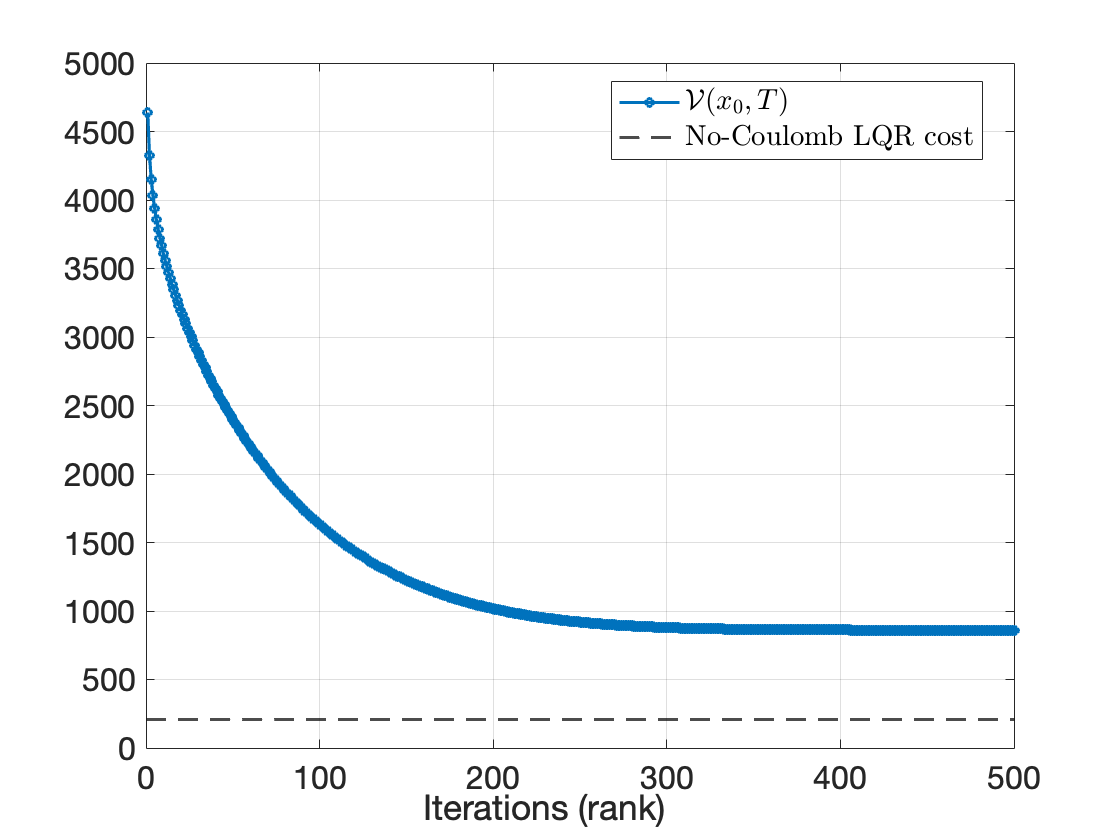}
		}
		\caption{The approximate value of the total action functional at initial states.}
		\label{fig:application1_value}
	\end{figure}
	
	In~\Cref{fig:application1_value}, we report the approximation of the value $\mathcal V(x_0,T)$ of the $N$-body system at the initial state $x_0$, as a function of the outer iteration number. 
	The curves are monotone, in agreement with \Cref{prop:monotone_lower_bounds} that each iteration enriches the max-plus representation by one additional supporting function and improves the lower bound of $v$, so the upperbound of $\mathcal V$. 
	The LQR baseline obtained by neglecting the Coulomb interaction is also shown for comparison.
	
\subsection{Random initial states with terminal cost}
	
	We also consider random initial states sampled in an annulus with inner and outer radii $(r_{\min},r_{\max})=(1,2)$, together with the same terminal cost as above. 
	A representative $30$-agent example is shown in~\Cref{fig:application3_traj}. 
	Compared with the symmetric case, the initial transient trajectory is no longer perfectly synchronized, but the same three-phase structure is observed: a quick organization phase, a long metastable annular regime, and a final contraction induced by the terminal cost. 
	The third panel again shows monotone improvement of the value at the initial state. 
	\begin{figure}[htb]
		\centering
		\subfloat[Trajectories for $N=30$.\label{fig:cas5_traj1}]{
                    \includegraphics[width=0.295\linewidth]{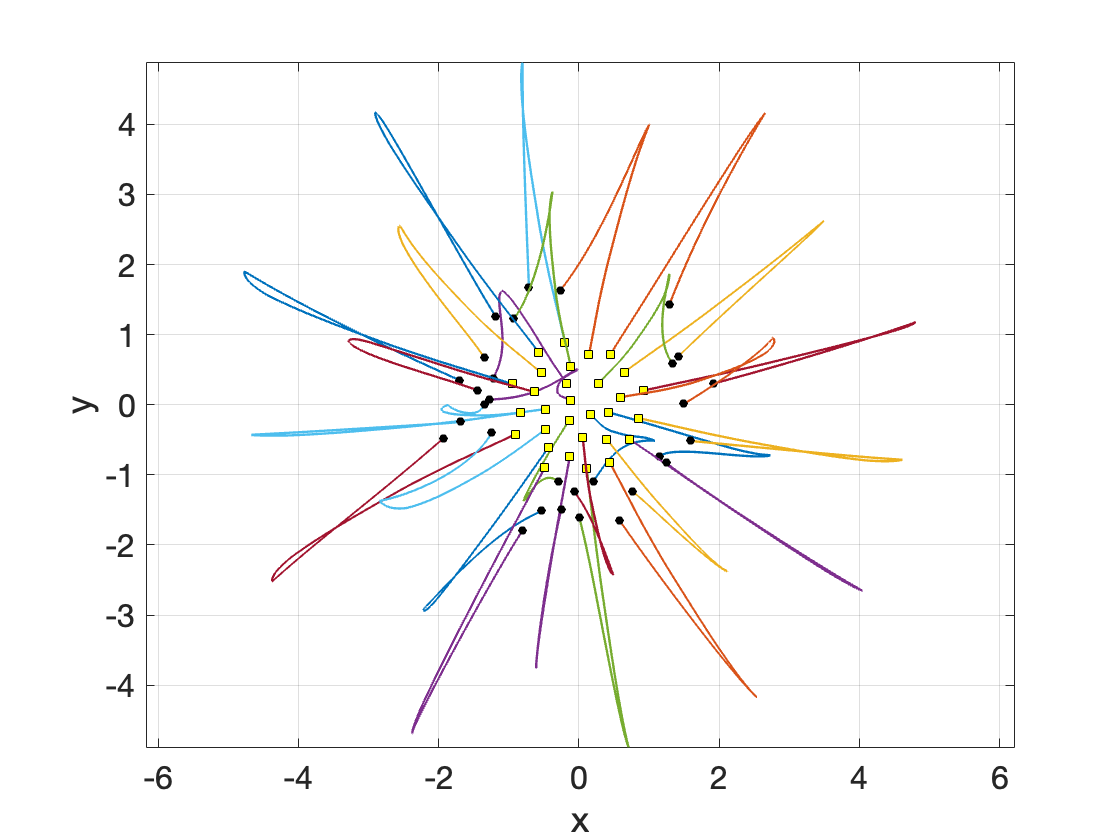}
		}
		\hfill
		\subfloat[Mean and maximum radii.\label{fig:case5_traj2}]{
			\includegraphics[width=0.295\linewidth]{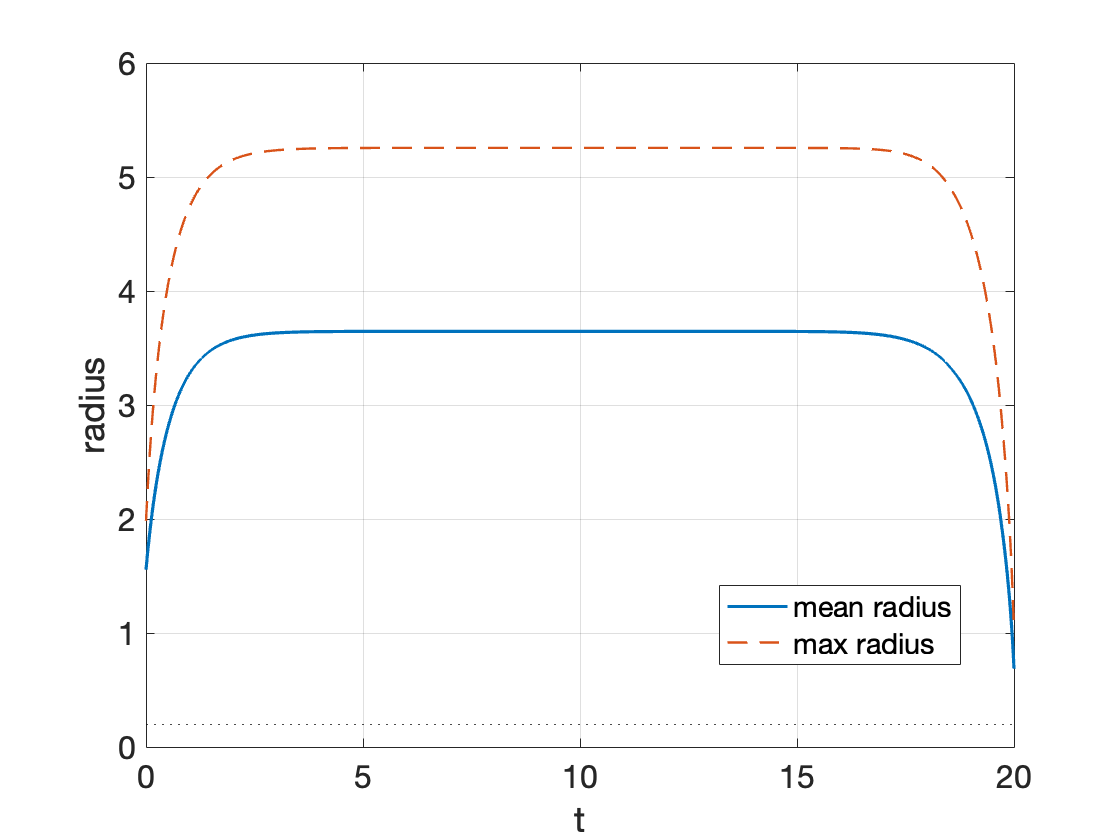}
		}
		\hfill
		\subfloat[Value of $\mathcal V(x_0,T)$.\label{fig:case5_value}]{
                  \includegraphics[width=0.295\linewidth,trim=0 0 0 2cm, clip]{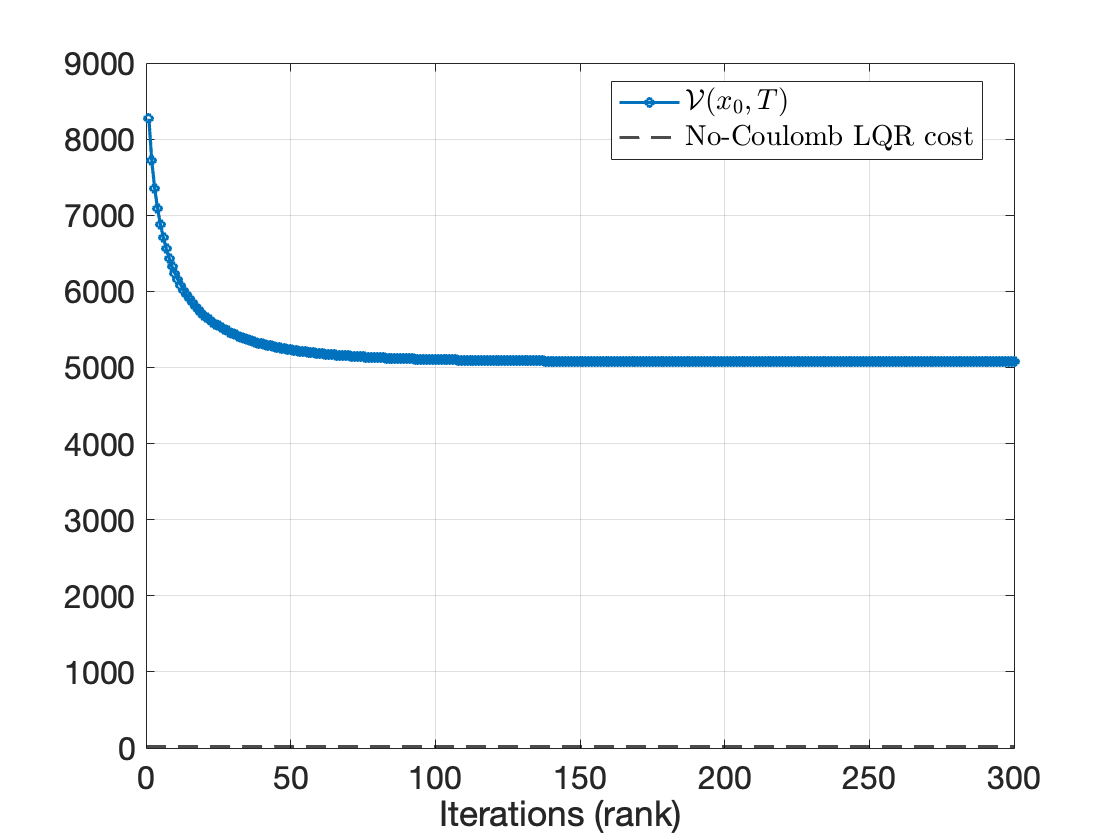}
		}
		\caption{Representative random-initialization experiment with $30$ agents and terminal cost.}
		\label{fig:application3_traj}
	\end{figure}

	\subsection{Random initial states and large-scale behavior}\label{sec-app2}
	
	We next consider a larger-scale random configuration. \Cref{fig:application4_traj} shows a run with $N=100$ agents and a terminal cost. 
	The first figure displays the trajectories, the second the evolution of the mean and maximum radii, and the third the value at the initial state as a function of the number of outer iterations. 
	The turnpike structure remains visible even at this scale, and the value curve remains monotone. 
	The radii plot indicates that, despite the random initialization, the particles rapidly concentrate near an effective annular steady state before the terminal cost forces the final inward motion. 
	In our implementation, this corresponds to a state dimension of $200$ and can be handled on a standard laptop with Apple M3 chip within a runtime budget near half hour. 
	This suggests that the method benefits strongly from exploiting separability in the representation, even though the underlying interaction is fully coupled.
	\begin{figure}[htb]
          	  \centering
          		\subfloat[Trajectories for $N=100$.\label{fig:cas6_traj1}]{
			\includegraphics[width=0.3\linewidth]{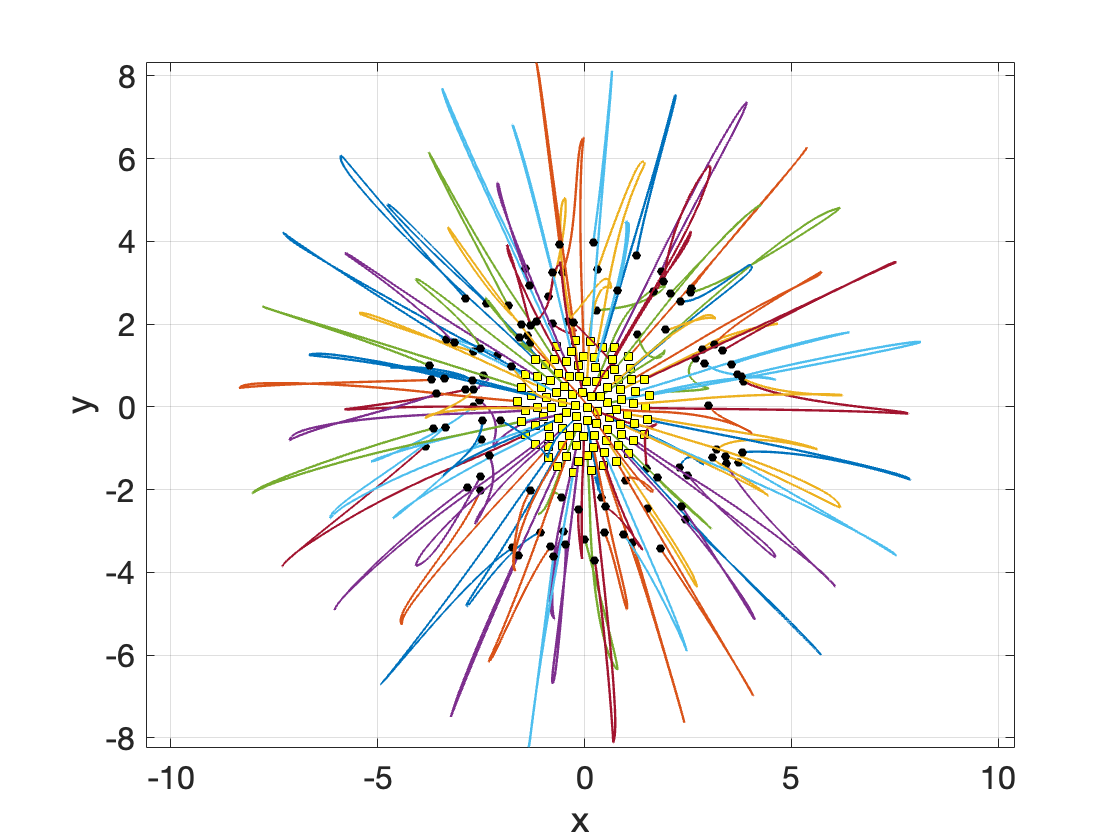}
		}
		\hfill
		\subfloat[Mean and maximum radii.\label{fig:case6_traj2}]{
			\includegraphics[width=0.3\linewidth]{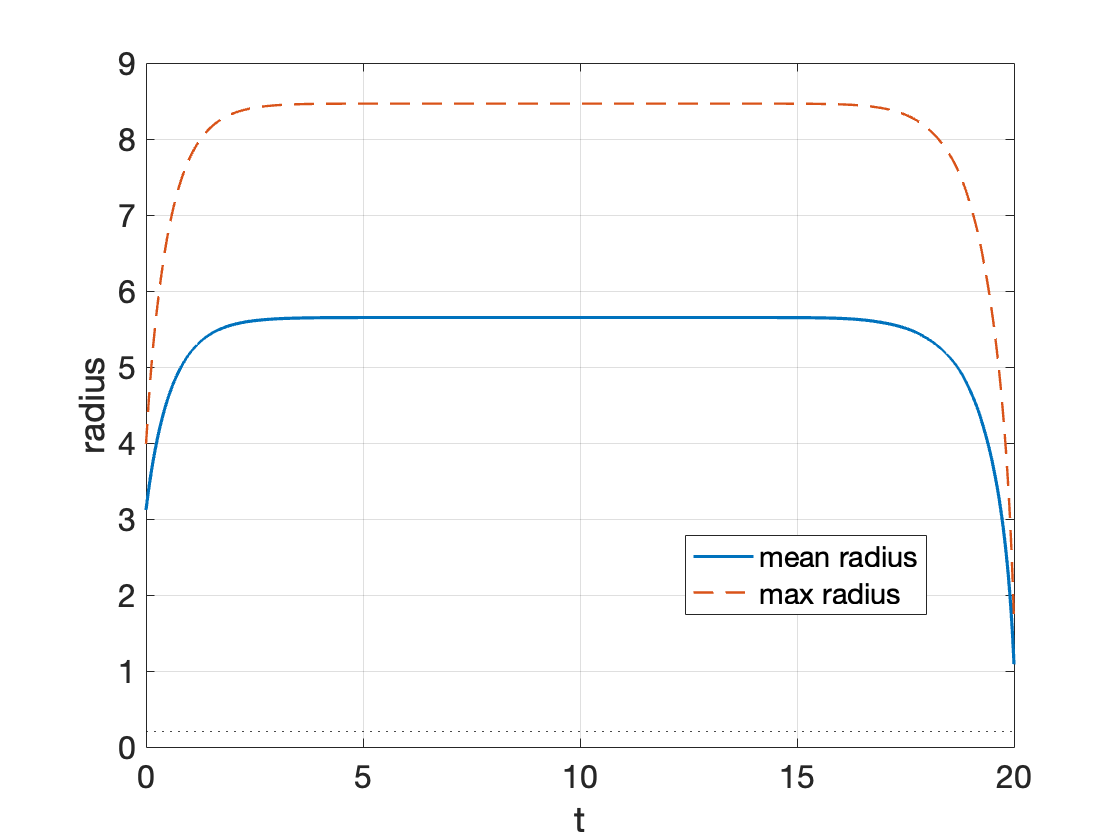}
		}
		\hfill
		\subfloat[Value of $\mathcal V(x_0,T)$. \label{fig:case6_value}]{
			\includegraphics[width=0.3\linewidth]{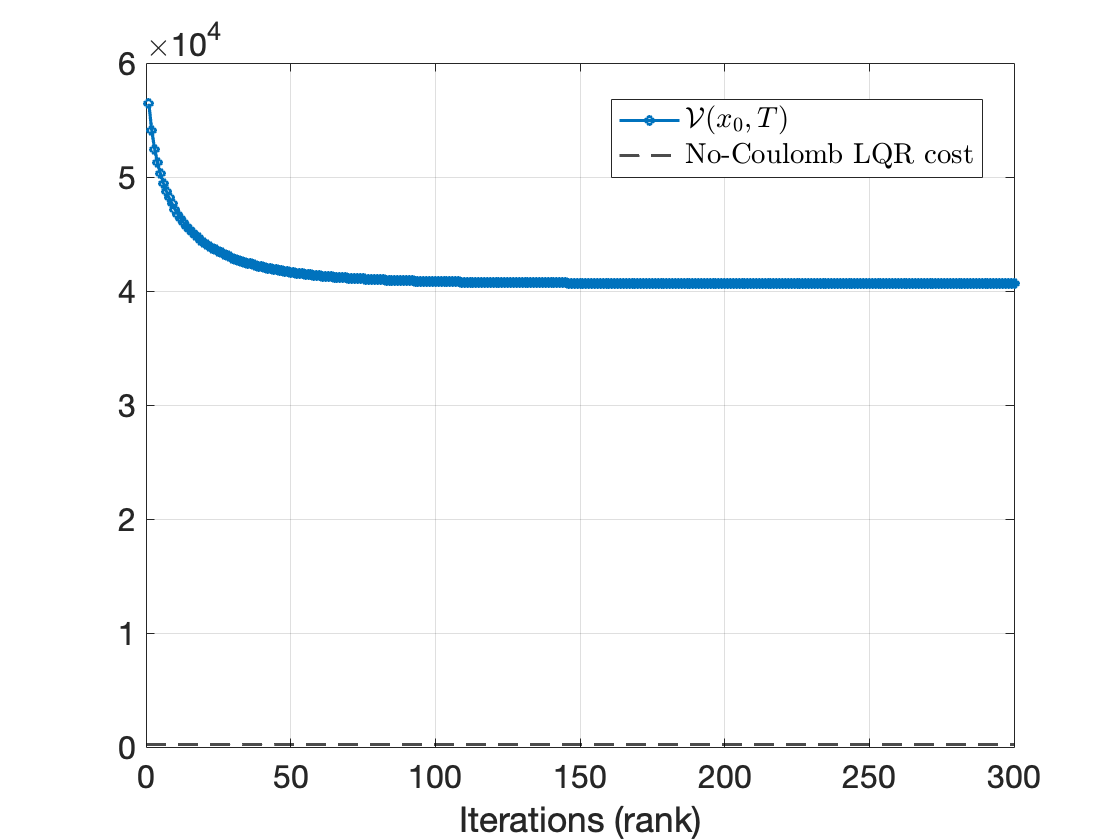}
		}
		\caption{Large-scale experiment with random initial positions and terminal cost.}
		\label{fig:application4_traj}
	\end{figure}

	
	\bibliographystyle{IEEEtran}
	\bibliography{ref}
	
\end{document}